\documentclass[reqno, 12pt]{amsart}

\usepackage{amsmath,amssymb,amsthm,amsfonts,graphicx, array}
\usepackage{cite}
\usepackage{fullpage}
\usepackage[colorlinks=true,linkcolor=blue,anchorcolor=blue,citecolor=red]{hyperref}
\usepackage{algorithm,algorithmic}
\renewcommand{\thealgorithm}{}

\newtheorem{theorem}{Theorem}
\newtheorem{lemma}{Lemma}

\newtheorem{definition}{Definition}

\newenvironment{nouppercase}{
  
  \renewcommand{\uppercasenonmath}[1]{}}{}

\allowdisplaybreaks

\author{Zhengkun Jia}
\address{School of Mathematical Sciences and LPMC, Nankai University, Tianjin 300071, China}
\email{Zhengkun Jia: zhengkun.jia.math@gmail.com}

\author{Huixi Li}
\address{School of Mathematical Sciences and LPMC, Nankai University, Tianjin 300071, China}
\email{Huixi Li: lihuixi@nankai.edu.cn}

\author{Yushuo Liu}
\address{College of  Computer Science, Nankai University, Tianjin 300071, China}
\email{Yushuo Liu: yushuo.liu@mail.nankai.edu.cn}

\date{\today}

\makeatletter
\@namedef{subjclassname@2020}{\textup{}2020 Mathematics Subject Classification}
\makeatother

\title[]{Resolving Adenwalla's conjecture related to a question of Erd\H{o}s and Graham about covering systems}

\begin{document}
	
\begin{abstract}
Erd\H{o}s and Graham posed the question of whether there exists an integer $n$ such that the divisors of $n$ greater than $1$ form a distinct covering system with pairwise coprime moduli for overlapping congruences. Adenwalla recently proved no such $n$ exists, introducing the concept of nice integers, those where such a system exists without necessarily covering all integers. Moreover, Adenwalla established a necessary condition for nice integers: if $n$ is nice and $p$ is its smallest prime divisor, then $n/p$ must have fewer than $p$ distinct prime factors. Adenwalla conjectured this condition is also sufficient. In this paper, we resolve this conjecture affirmatively by developing a novel constructive framework for residue assignments. Utilizing a hierarchical application of the Chinese Remainder Theorem, we demonstrate that every integer satisfying the condition indeed admits a good set of congruences. Our result completes the characterization of nice integers, resolving an interesting open problem in combinatorial number theory. 
\end{abstract}

\begin{nouppercase}
\maketitle
\end{nouppercase}

\section{Introduction}\label{Intro}

A \emph{covering system} of $\mathbb{Z}$ is a set of congruences $\{a_1 \pmod{d_1}, \cdots, a_t \pmod{d_t}\}$, where $1 < d_1 \leq \cdots \leq d_t$, that covers all integers. We say a covering system is \emph{distinct} if the moduli of the congruences satisfy $1 < d_1 < \cdots < d_t$. 

The study of distinct covering systems has been a cornerstone of combinatorial number theory since the foundational work of Erd\H{o}s \cite{Erdos1950}. There are a lot of progress in this area. With respect to the least residue problem proposed by Erd\H{o}s \cite{Erdos1950} in 1950, 
Hough \cite{Hough2015} proved the least residue in a distinct covering system is at most $10^{16}$ in 2015, Balister, Bollob\'{a}s, Morris, Sahasrabudhe, and Tiba \cite[Theorem 8.1]{BBMST2022} reduced this bound to $616000$ in 2022, and Cummings, Filaseta and Trifonov \cite{CFT2025} reduced this bound to $118$ for distinct covering systems with squarefree moduli in 2025. With respect to the divisibility relation among the moduli of a distinct covering system asked by Schinzel \cite{Schinzel1967} in 1967, Balister, Bollob\'{a}s, Morris, Sahasrabudhe, and Tiba proved that in a distinct covering system, at least one of the moduli divides another \cite[Theorem 1.2]{BBMST2022}. With respect to the question asked by Erd\H{o}s and Selfridge \cite[Section F13]{Guy2004} in the 1950s about whether all moduli in a distinct covering system can be odd, this question is still open. Recent work by \cite{HN2019, BBMST2021, BBMST2022} has introduced critical and significant constraints on potential solutions. 

Among the many interesting problems in this field, a question posed by Erd\H{o}s and Graham \cite{EG1980} asks in 1980 states that ``Does there exist an integer $n$ such that the divisors of $n$ greater than $1$ form a distinct covering system where any overlapping congruences satisfy $\gcd(d, d') = 1$?" This is listed as Problem 208 in Bloom’s list of Erd\H{o}s' problems \cite{Bloom2025}. Recently, Adenwalla~\cite{Adenwalla2025} has proved there are no such $n$. To explain Adenwalla's approach, we introduce the following four definitions. 

\begin{definition}
Two congruences $a \pmod{b}$ and $a' \pmod{b'}$ \emph{overlap} if there is an integer $x$ such that $x \equiv a \pmod{b}$ and $x \equiv a' \pmod{b'}$. 
\end{definition}

\begin{definition}
We say a set of congruences, $\{a_1 \pmod{d_1}, \cdots, a_t \pmod{d_t} \}$, where $1 \leq d_1 < \cdots < d_t$, is \emph{good}, if whenever two congruences $a_i \pmod{d_i}$ and $a_j \pmod{d_j}$ in the set overlap for some $1 \leq i \neq j \leq n$, we have $\gcd(d_i, d_j) = 1$. 
\end{definition}

\begin{definition}
We say a positive integer $n$ is \emph{nice} if there exist integers $\{a_d: d \mid n, d > 1\}$ such that $\{a_d \pmod{d}: d|n,d > 1\}$ is a good set of congruences with respect to $n$. 
\end{definition}

\begin{definition}
We say a positive integer $n$ is \emph{admissible} if there exists a good set of congruences with respect to $n$ which is also a covering system. 
\end{definition}

With these definitions, Erd\H{o}s and Graham's question can be rephrased as ``Does an admissible integer $n$ exist"? For admissible integers $n$, by M. Newman's proof that the sum of reciprocal of the moduli in a distinct covering system must exceed $1$, we know $\sigma(n)/n > 2$, where $\sigma(n)$ is the sum of divisors of $n$. Benkoski and Erd\H{o}s \cite{BE1974} wondered in 1974 that if $\sigma(H)/H$ were large enough, would this condition imply $H$ is admissible. Haight \cite{Haight1979} demonstrated in 1979 a striking negative result: there exist integers $H$ with $\sigma(H)/H$ arbitrarily large that are not admissible. Haight's result was strengthened by Filaseta, Ford, Konyagin, Pomerance, and Yu \cite{FFKPY2007} in 2007, by a shorter proof. Their approach constructs $H$ as a product of primes in intervals like $(N, N^K]$, showing that despite $\sigma(H)/H \sim (log log H)^{1/2}$, the density of residues not covered corresponding to moduli that are $H$’s divisors which are bigger than $1$ remains positive. We refer the interested readers to \cite[Theorem 1]{FFKPY2007}. 

Recently in 2025, through a sophisticated analysis of density bounds and structural constraints on prime factorizations, Adenwalla \cite{Adenwalla2025} has shown that any set of congruences with moduli given by the divisors of $n$ greater than $1$ necessarily fails to cover all integers, and therefore no admissible integers exist. 

Clearly, admissible integers are nice integers. Therefore, a natural question emerged: Which integers are nice? Adenwalla proved numbers of the form $p^k$ for a prime $p$ and an integer $k \geq 1$ are nice, and numbers of the form $p_1 p_2^k$ for primes $p_1$ and $p_2$ and an integer $k \geq 1$ are nice. Moreover, he established a necessary condition for nice integers. 
\begin{theorem}[Adenwalla, {\cite[Lemma 3.1]{Adenwalla2025}}]
Let $n > 1$ be nice. If $p$ is the smallest prime divisor of $n$, then $\frac{n}{p}$ must have fewer than $p$ distinct prime factors.
\end{theorem} 
Furthermore, Adenwalla conjectured that the above theorem also gives a sufficient condition for nice integers \cite[Conjecture 5.1]{Adenwalla2025}. 

In this paper, by developing a novel constructive framework for residue assignments and leveraging the Chinese Remainder Theorem in a non-uniform way, we demonstrate that every integer $n$ satisfying the condition in the above theorem indeed forms a good congruence set, i.e., we resolve Adenwalla's conjecture in the affirmative. 

\begin{theorem}\label{mainthm}
Let $p$ be the smallest prime divisor of $n$. If $\frac{n}{p}$ has less than $p$ distinct prime divisors, then $n$ is nice.     
\end{theorem}

Our main theorem extends beyond the special cases, e.g., prime and biprime powers,  treated in \cite{Adenwalla2025}, generalizing to arbitrary composite structures under the required condition. The proof of the theorem is by explicit construction of good sets of congruences. Our approach introduces a unified combinatorial strategy, the core innovation lies in our hierarchical residue assignment method, which systematically constructs sets of  congruence by leveraging the Chinese Remainder Theorem in a non-uniform, recursively structured manner. By decoupling residue assignments across distinct prime power layers, we break through the rigidity of traditional uniform applications of the Chinese Remainder Theorem, enabling flexible yet conflict-free constructions.

The paper is structured as follows. In Section 2 we establish some lemmas on divisibility properties of moduli for nice numbers and algebraic conditions for residue assignments. In Section 3 we apply the hierarchical residue assignment method to prove our main theorem. In Section 4 we provide two concrete examples of nice numbers to help the readers better understand our explicit construction and one counterexample with $\omega(n/p) = p$ to demonstrate where the barrier occurs. 

\noindent \textbf{Notation.} For a positive integer $n \geq 2$ we let $\omega(n)$ be the number of distinct prime divisors of $n$. 

\section{Proof of some lemmas}

In this section we establish the key combinatorial and algebraic tools necessary for the constructive proof of Theorem 2. We first prove two general lemmas on the inheritance of ``niceness" under divisibility and the stability of good sets under union operations. These results unify the treatment of prime powers and composite moduli in later sections. 

\begin{lemma}\label{lemdivisor}
If $n$ is nice, and $m \mid n$, then $m$ is also nice. 
\end{lemma}

\begin{proof}
This follows from the definition of nice numbers. 
\end{proof}

\begin{lemma}\label{lemmangood}
Let $t \geq 3$. If $A_1$, $A_2$, $\cdots$, $A_t$ are $t$ good sets of congruences, such that $A_i \cup A_j$ is still good for all $1 \leq i < j \leq t$, we have $\mathop{\bigcup}\limits_{i=1}^t A_i$ is good. 
\end{lemma}
\begin{proof}
Suppose $\mathop{\bigcup}\limits_{i=1}^t A_i$ is not good, then there exist two congruences $a \pmod{b}\in A_u$ and $a' \pmod{b'} \in A_v$ that overlap with $\gcd(b, b')>1$ for some $1 \leq u \leq v \leq t$. Since each $A_i$, $1 \leq i \leq t$, is good, we have $u \neq v$. However, this contradicts $A_u\bigcup A_v$ is good, which completes the proof of the lemma. 
\end{proof}

We now turn to explicit constructions of good sets of congruences for specific classes of integers, which will form the building blocks for our recursive framework in Section 3. The following 5 lemmas demonstrate how hierarchical residue assignments can be systematically designed to satisfy the ``goodness" criterion, even for composite moduli with intricate prime factorizations. 

\begin{lemma}\label{lempk}
Let $n=p^k$ for a prime $p$ and $k \geq 1$. Then $n$ is nice.
\end{lemma}

\begin{proof}
This proved in Prop 4.1. We can explicitly construct the following good set of congruences with respect to $n$: 
\[
\{p^{i-1} \pmod{p^i}: 1 \leq i \leq k\}.
\]
\end{proof}

\begin{lemma}\label{lemqpk}
Let $n = p_1 p_2^k$, where $p_1 < p_2$ are primes and $k \geq 1$. Then $n$ is nice.  
\end{lemma}

\begin{proof}
This is proved in \cite[Proposition 4.2]{Adenwalla2025}. We can explicitly construct the following good set of congruences with respect to $n$: 
\[
\{0\pmod{p_1}\}
\mathop{\bigcup}\limits_{1\leq i\leq k}\{p_2^{i-1}+1\pmod{p_2^i}\}
\mathop{\bigcup}\limits_{1\leq i\leq k}\{{a p_1 p_2^{i-1}+1\pmod{p_1p_2^i}}\}, 
\]
where $a$ satisfies $a\not\equiv 0$, $p_1^{-1}\pmod{p_2}$. 

\end{proof}

\begin{lemma}\label{lemmatwofactor}
Let $p_1 < p_2$ be odd primes. Let $a$ be a given integer. Let $b$, $c$, and $d$ be three different residues modulo $p_2$. Let $\alpha_1 \geq 1$ and $\alpha_2 \geq 1$. Then there exists a good set of congruences $A=\left\{a_{p_1^i p_2^j} \pmod{p_1^i p_2^j}:1 \leq i \leq \alpha_1, 1 \leq j \leq \alpha_2\right\}$, such that $a_{p_1^i p_2^j}$ satisfies that $a_{p_1^i p_2^j} \equiv a \pmod {p_1}$ and $a_{p_1^i p_2^j} \pmod{p_2} \in \{b, c, d\}$. 
\end{lemma}
\begin{proof}
We select the following set of congruences

\begin{align*}
A=&
\left\{
a_{p_1 p_2} \pmod{p_1 p_2}\right\}
\bigcup_{ 2 \leq i \leq \alpha_2}
\left\{a_{p_1p_2^i} \pmod{p_1p_2^i}\right\} \\
& \bigcup_{ 2 \leq i \leq \alpha_1}
\left\{a_{p_1^ip_2} \pmod{p_1^ip_2}\right\}\bigcup_{2 \leq i \leq \alpha_1, 2 \leq j \leq \alpha_2}
\left\{a_{p_1^ip_2^j} \pmod{p_1^ip_2^j}\right\}, 
\end{align*}
in which 
\[
\begin{cases}
a_{p_1 p_2} \equiv a \pmod{p_1}, \\
a_{p_1 p_2} \equiv b \pmod{p_2}, 
\end{cases}
\begin{cases}
a_{p_1^ip_2} \equiv \begin{cases} 
a+p_1\pmod{p_1^2},   & \text{if } i = 2, \\
a+p_1^{i-1}\pmod{p_1^{i}}, & \text{if } i\geq3, 
\end{cases}\\ 
a_{p_1^ip_2}  \equiv c \pmod{p_2}, 
\end{cases}
\]
\[
\begin{cases}
a_{p_1p_2^i} \equiv a \pmod{p_1}, \\
a_{p_1p_2^i} \equiv 
\begin{cases} 
d+p_2\pmod{p_2^2},   & \text{if } i = 2, \\
d+p_2^{i-1}\pmod{p_2^{i}}, & \text{if } i\geq3, 
\end{cases}
\end{cases}
\begin{cases}
a_{p_1^ip_2^j}\equiv \begin{cases} 
a+p_1\pmod{p_1^2},   & \text{if } i = 2, \\
a+p_1^{i-1}\pmod{p_1^{i}}, & \text{if } i\geq3, 
\end{cases} \\ \\
a_{p_1^ip_2^j} \equiv\begin{cases} 
d+2p_2\pmod{p_2^2},   & \text{if } i = 2, \\
d+3p_2+p_2^{j-1}\pmod{p_2^{j}}, & \text{if } i\geq3.
\end{cases} 
\end{cases}
\]
Next we prove $A$ is good. Suppose $A$ is not good, then there exist two different congruences $a_{p_1^{i_1}p_2^{i_2}} \pmod{p_1^{i_1}p_2^{i_2}}$ and $a_{p_1^{j_1}p_2^{j_2}} \pmod{p_1^{j_1}p_2^{j_2}}$ in $A$ that overlap with $(i_1,i_2)\neq(j_1,j_2)$. Since $a_{p_1^{i_1}p_2^{i_2}} \pmod{p_1^{i_1}p_2^{i_2}}$ and $a_{p_1^{j_1}p_2^{j_2}} \pmod{p_1^{j_1}p_2^{j_2}}$ overlap, we have
\[a_{p_1^{i_1}p_2^{i_2}}\equiv a_{p_1^{j_1}p_2^{j_2}}\pmod{p_1^{min\{i_1,j_1\}}}\]
and 
\[a_{p_1^{i_1}p_2^{i_2}}\equiv a_{p_1^{j_1}p_2^{j_2}}\pmod{p_2^{min\{i_2,j_2\}}}.\]

If $a_{p_1^{i_1}p_2^{i_2}}\equiv a_{p_1^{j_1}p_2^{j_2}}\equiv b\pmod{p_2}$, then by our construction we know $i_1=j_1$, which is a contradiction. 

If $a_{p_1^{i_1}p_2^{i_2}}\equiv a_{p_1^{j_1}p_2^{j_2}}\equiv c\pmod{p_2}$, then by our construction, we know $i_1,j_1 \geq2$ and $i_2=j_2=1$. 
Thus $a_{p_1^{i_1}p_2^{i_2}}\equiv a_{p_1^{j_1}p_2^{j_2}}\pmod{p_1^2}$, which implies either $i_1=j_1=2$, or $i_1,j_1\geq3$, which again implies $i_1=j_1$, a contradiction.  

If $a_{p_1^{i_1}p_2^{i_2}}\equiv a_{p_1^{j_1}p_2^{j_2}}\equiv d\pmod{p_2}$, we note that $i_2,j_2\geq2$, since $a_{p_1^{i_1}p_2^{i_2}}\equiv a_{p_1^{j_1}p_2^{j_2}}\pmod{p_2^{min\{i_2,j_2\}}}$, $a_{p_1p_2^{i_2}}\pmod{p_2^2}\in\{d,d+p_2\}$, and $a_{p_1^{i_1}p_2^{i_2}}\pmod{p_2^2}\in\{d+2p_2,d+3p_2\}$, we have 
$i_1=j_1=1$ or $i_1,j_1\geq2$. If $i_1=j_1=1$, then we have $a_{p_1p_2^{i_2}}\pmod{p_2^2}\in\{d,d+p_2\}$, similar to the above discussion, we have $i_2=j_2$, a contradiction. If $i_1,j_1\geq2$, we have $a_{p_1^{i_1}p_2^{i_2}}\pmod{p_2^2}\in\{d+2p_2,d+3p_2\}$ and $a_{p_1^{i_1}p_2^{i_2}}\pmod{p_1^2}\in\{a+p_1,a\}$, similar to the above discussion, we have $i_1=j_1$ and $i_2=j_2$, a contradiction. 

Therefore, the set $A$ we have constructed is good. 
\end{proof}

\begin{lemma}\label{lemp_1alpha_1p_2alpha_2}
Let $n=p_1^{\alpha_1} p_2^{\alpha_2}$, where $p_1 < p_2$ are odd primes and $\alpha_1 > 1$, $\alpha_2 > 1$. Then $n$ is nice.
\end{lemma}
\begin{proof}
By our assumption we have $p_1 \geq 3$ and $p_2 \geq 5$. We can explicitly construct the following good set of congruences with respect to $n$. Let 
\begin{align*}
A=&
\{ 1 \pmod{p_1} \} 
\bigcup_{2 \leq i \leq \alpha_1} 
\left\{p_1^{\,i-1} \pmod{p_1^i} \right\}
\bigcup 
\{1 \pmod{p_2}\}
\bigcup_{2 \leq i \leq \alpha_2} 
\left\{ p_2^{\,i-1} \pmod{ p_2^i}\right\} \\
& \left\{
a_{p_1 p_2} \pmod{p_1 p_2}\right\}
\bigcup_{ 2 \leq i \leq \alpha_2}
\left\{a_{p_1p_2^i} \pmod{p_1p_2^i}\right\} \bigcup_{ 2 \leq i \leq \alpha_1}
\left\{a_{p_1^ip_2} \pmod{p_1^ip_2}\right\}\\
&\bigcup_{2 \leq i \leq \alpha_1, 2 \leq j \leq \alpha_2}
\left\{a_{p_1^ip_2^j} \pmod{p_1^ip_2^j}\right\},  
\end{align*}
and apply Lemma~\ref{lemmatwofactor} with $a=2$, $b=2$, $c=3$, and $d=4$, we know that $A$ is good.
\end{proof}

\begin{lemma}\label{lemmanfactor}
Let $t \geq 3$. Let $p_1 < p_2 < \cdots < p_t$ be odd primes. Let $\alpha_1 \geq 1$, $\alpha_2 \geq 1$, $\cdots$, $\alpha_t \geq 1$. If a set $S$ of residues modulo $p_1 p_2 \cdots p_t$ has cardinality at least $2^t$, then there exist a good set of congruences 
\[
A=
\left\{
a_{i_1,i_2,\cdots ,i_t} \pmod {p_1^{i_1} p_2^{i_2} \cdots p_t^{i_t}}: 1 \leq i_k \leq \alpha_k, 1 \leq k \leq t 
\right\}, 
\]
such that $a_{i_1, i_2, \cdots, i_t} \in S$, where $1 \leq i_k \leq \alpha_k$ and $1 \leq k \leq t$. 
\end{lemma}
\begin{proof}
Let $a_0, a_1, \cdots, a_{2^t-1} \in S$. For the exponents we let 
\[ 
e(m) =
\begin{cases} 
1,   & \text{if } m = 1, \\
m-1, & \text{if } m \geq 2, 
\end{cases}
\]
and for the sub-indexes we let 
\[
s(m)=
\begin{cases}
0, & \text{if } m = 1, \\
1, & \text{if } m \geq 2. 
\end{cases}
\]
Then, we explicitly construct a set of congruences 
\[
A = 
\left\{
p_1^{e(i_1)} p_2^{e(i_2)} \cdots p_n^{e(i_t)} + a_{s(i_1)+2s(i_2)+4s(i_3)+\cdots+2^{n-1}s(i_t)}
\pmod{p_1^{i_1}\,p_2^{i_2}\cdots p_t^{\,i_t}}: 1 \leq i_k \leq \alpha_k, 1 \leq k \leq t
\right\}. 
\]
Note that since $0 \leq s(m) \leq 1$, we have $0 \leq s(i_1) + 2s(i_2) + 4s(i_3) + \cdots + 2^{t-1} s(i_t) \leq 2^t - 1$. 

Next we prove $A$ is good. Suppose $A$ is not good, then there exist $t$-tuples $(i_1, i_2, \cdots, i_t) \neq (j_1, j_2, \cdots, j_t)$ such that their corresponding congruences in $A$ overlap. Then there exists an integer $x$ that satisfies the following two congruences: 
\begin{equation}\label{xi}
x \equiv p_1^{e(i_1)} p_2^{e(i_2)} \cdots p_t^{e(i_t)} + a_{s(i_1) + 2s(i_2) + 4s(i_3) + \cdots + 2^{t-1}s(i_t)} \pmod{p_1^{i_1} p_2^{i_2} \cdots p_t^{i_t}}     
\end{equation}
and 
\begin{equation}\label{xj}
x \equiv p_1^{e(j_1)} p_2^{e(j_2)} \cdots p^{e_t(j_t)} + a_{s(j_1) + 2s(j_2) + 4s(j_3) + \cdots + 2^{t-1}s(j_t)} \pmod{p_1^{j_1} p_2^{j_2} \cdots p_t^{j_t}}.     
\end{equation}
Then we have 
\[
a_{s(i_1)+2s(i_2)+4s(i_3)+\cdots+2^{t-1} s(i_t)}\equiv a_{s(j_1)+2s(j_2)+4s(j_3)+\cdots+2^ts(j_t)}\pmod {p_1p_2\cdots p_t}. 
\] 
Hence 
\[
s(i_1)+2s(i_2)+4s(i_3)+\cdots+2^{t-1}s(i_t)=s(j_1)+2s(j_2)+4s(j_3)+\cdots+2^{t-1}s(j_t). 
\]
Since $0 \leq s(m) \leq 1$, for all $1 \leq k \leq t$ we have 
\begin{equation}\label{sijk}
s(i_k)=s(j_k).    
\end{equation}
To simplify our notation we denote $a_{s(i_1)+2s(i_2)+4s(i_3)+\cdots+2^{t-1} s(i_t)}$ by $a$, then equations~\eqref{xi} and \eqref{xj} can be rewritten as
\begin{equation}\label{xieasy}
x \equiv p_1^{e(i_1)} p_2^{e(i_2)} \cdots p_t^{e(i_t)} + a \pmod{p_1^{i_1} p_2^{i_2} \cdots p_t^{i_t}}     
\end{equation}
and 
\begin{equation}\label{xjeasy}
x \equiv p_1^{e(j_1)} p_2^{e(j_2)} \cdots p^{e_t(j_t)} + a \pmod{p_1^{j_1} p_2^{j_2} \cdots p_t^{j_t}}.     
\end{equation}
By Equations~\eqref{xieasy} and \eqref{xjeasy} we know for all $1 \leq k \leq t$ we have 
\[
p_{k}^{e(i_k)} \mid x-a, \quad p_{k}^{e(i_k)+1} \nmid x-a,
\]
\[p_{k}^{e(j_k)} \mid x-a, \quad p_{k}^{e(j_k)+1} \nmid x-a,
\]
Hence, for all $1\leq k\leq t$ we have
\begin{equation}\label{eijk}
e(i_k)=e(j_k).     
\end{equation}
By Equations~\eqref{sijk}, ~\eqref{eijk}, and the definitions of $e(m)$ and $s(m)$, it follows that $i_k=j_k$ for all $1 \leq k \leq t$. 
This is a contradiction. Hence $A$ is good.
\end{proof}

\section{Proof of Theorem~\ref{mainthm}}

In this section we prove our main theorem. By Lemmas~\ref{lempk}, \ref{lemqpk}, and \ref{lemp_1alpha_1p_2alpha_2}, we only need to consider integers $n$ with $\omega(n) = t \geq 3$. Let $p_1$ be the smallest prime factor of $n$, and let $\alpha_1$ be the exponent of $p_1$ in $n$. Next we discuss the following two cases, based on whether $\alpha_1 \geq 2$ or $\alpha_1 = 1$. 

\noindent \textbf{Case 1.} Let $n = p_1^{\alpha_1} p_2^{\alpha_2} \cdots p_t^{\alpha_t}$ with $p_1 < p_2 < \cdots < p_t$, where $\alpha_1 \geq 2$, and $\alpha_k \geq 1$ for $2 \leq k \leq t$. Let $t \leq p_1-1$. By Lemma~\ref{lemdivisor}, Theorem~\ref{mainthm} holds true if we can prove $n$ is still nice if we further assume $\alpha_k \geq 2$ for all $1 \leq k \leq t$. 

To prove $n$ is nice, we first construct good sets of congruences $A_{(1)}$, $A_{(2)}$, $\cdots$, $A_{(t)}$, and then prove $\mathop{\bigcup}\limits_{k = 1}^t A_{(k)}$ is still a good set of congruences. The set $A_{(k)}$, where $1 \leq k \leq t$, contains the congruences with moduli $m$ such that $m \mid n$ and $\omega(m) = k$. 

We first construct a good set of congruences $A_{(1)}$, in which the moduli have only one prime factor. Let 
\[
A_{(1)}=
\bigcup_{1 \leq i \leq t}
\{1 \pmod {p_i}\} 
\bigcup_{1 \leq i \leq t, 2 \leq \beta \leq \alpha_i} 
\left\{p_i^{\beta-1} \pmod {p_i^{\beta}}\right\}. 
\]
Similar to the proof of Lemma~\ref{lempk}, the set $A_{(1)}$ is good. Observe that for any $1 \leq i \leq t$, any congruence $a \pmod{p_i^\gamma}$ in $A_{(1)}$, where $1 \leq \gamma \leq \alpha_i$, satisfies $a \equiv 0$ or $1 \pmod{p_i}$. 

Now, we construct the set of congruences $A_{(2)}$. We start by constructing some subsets $A_{i,j}$ of $A_{(2)}$, where $1 \leq i < j \leq t$. Let 
\[
A_{i,j}=
\bigcup_{1 \leq \beta_1 \leq \alpha_i, 1 \leq \beta_2 \leq \alpha_j}
\left\{a_{ p_i^{\beta_1} p_j^{\beta_2}}\pmod{
p_i^{\beta_1} p_j^{\beta_2}}\right\}. 
\]
We require
\begin{equation}\label{modpi}
a_{ p_i^{\beta_1} p_j^{\beta_2}}\equiv j+2(i-1)\,\pmod{p_i}
\end{equation}
and 
\begin{equation}\label{modpj}
a_{ p_i^{\beta_1} p_j^{\beta_2}} \pmod{p_j} \in \{3i-1,3i,3i+1 \}. 
\end{equation}
We set
\[
A_{(2)}=\bigcup_{1 \leq i<j \leq t}A_{i,j}. 
\]
With the restriction \eqref{modpi} and \eqref{modpj}, we can observe that for all given $k$ such that $1 \leq k  \leq t$, any congruence $a_{p_k^{\beta_k} p_\ell^{\beta_\ell}} \pmod{p_k^{\beta_k} p_l^{\beta_\ell}}$ in $A_{(2)}$, where $1 \leq \ell \leq t$ such that $\ell \neq k$, $1 \leq \beta_k \leq \alpha_k$ and $1 \leq \beta_\ell \leq \alpha_\ell$, satisfies $a_{p_k^{\beta_k} p_l^{\beta_l}}\pmod{p_k}$ ranges over $\{ 3k - 1, 3k, \cdots, t + 2(k - 1) \}$ for $k<l\leq t$ and $a_{p_k^{\beta_k} p_l^{\beta_{\ell}}}\pmod{p_k}$ ranges over $\{2, 3, \cdots,  3k - 2\}$ for $1\leq l <k$, thus $a_{p_k^{\beta_k} p_l^{\beta_{\ell}}} \pmod{p_k}$ ranges over $\{2, 3, \cdots, t+2(k-1)\}$. Since $t \leq p_1 - 1$, we have $t+2(k-1)\leq p_1-1+2(k-1)\leq p_k-1$, and thus the residues $a_{p_k^{\beta_k} p_l^{\beta_{\ell}}}\pmod{p_k}$ satisfy $\{2, 3, \cdots, t+2(k-1)\} \subset \{0, 1, \cdots, p_k - 1\}$. By Lemma~\ref{lemmatwofactor} with $a = j + 2(i-1)$, $b = 3i - 1$, $c = 3i$, and $d = 3i + 1$, there exists a good set of congruences $A_{i, j}$ satisfying Equations~\eqref{modpi} and ~\eqref{modpj}. 
    
We now prove that $A_{(1)}\bigcup A_{(2)}$ is good. By Lemma~\ref{lemmangood}, it suffices to show that the unions $A_{(1)} \bigcup A_{i,j}$, where $1 \leq i < j \leq t$, and $A_{i,j} \bigcup A_{k,l}$ are good, where $(i,j)\neq(k,\ell)$, $1 \leq i < j \leq t$ and $1 \leq k< \ell \leq t$.
    
For $A_{(1)} \bigcup A_{i,j}$, since $A_{(1)}$ and $A_{i,j}$ are both good, if $A_{(1)} \bigcup A_{i,j}$ is not good, there exist two the congruences 
\[
a_{p^{\beta_1}} \pmod{p^{\beta_1}}\in A_{(1)}
\]
and
\[
a_{p_{i}^{\beta_i}p_{j}^{\beta_j}} \pmod{ p_{i}^{\beta_i}p_{j}^{\beta_j}}\in A_{i,j}
\]
that overlap and $\gcd(p^{\beta_1}, p_{i}^{\beta_i} p_{j}^{\beta_j}) > 1$. Without loss of generality, we can assume $p = p_{i}$. Note that
\[
a_{p^{\beta_1}} \equiv 0, 1 \pmod{p}
\]
and
\[
a_{p_{i}^{\beta_i} p_{j}^{\beta_j}}\pmod{p_i^{\beta_i}}\in\{2, 3, \cdots, t+2(k-1)\},   
\]
this contradicts that $a_{p^{\beta_1}} \pmod{p^{\beta_1}}$ and $a_{p_{i}^{\beta_i}p_{j}^{\beta_j}} \pmod{ p_{i}^{\beta_i}p_{j}^{\beta_j}}$ overlap. 

For $A_{i,j}\bigcup A_{k,\ell}$, since $A_{i,j}$ and $A_{k,\ell}$ are both good, if $A_{i,j} \bigcup A_{k,\ell}$ is not good, then there exist two  congruences 
\[
a_{p_{i}^{\beta_i}p_{j}^{\beta_j}} \pmod{ p_{i}^{\beta_i}p_{j}^{\beta_j}}\in A_{i,j}
\]
and
\[ a_{p_{k}^{\beta_{k}}p_{\ell}^{\beta_{\ell}}}\pmod{p_{k}^{\beta_{k}}p_{\ell}^{\beta_{\ell}}}\in A_{k,\ell}
\]
that overlap with $\gcd(p_{i}p_{j}, p_{k}p_{l}) > 1$. If $p_i = p_k$, then $i=k$. Since 
\[
a_{p_{i}^{\beta_i}p_{j}^{\beta_j}}\equiv j+2(i-1)\pmod{p_i},
\]
and 
\[
a_{p_{k}^{\beta_{k}}p_{\ell}^{\beta_{\ell}}}\equiv l+2(k-1)\pmod{p_k}
\]
overlap, note that $j$ and $\ell$ satisfy $j \leq t < p_i$ and $\ell \leq t < p_i$ we know $j=\ell$, this is a contradiction.

If $p_i=p_\ell$, we have $i=\ell$. Since 
\[
a_{p_{i}^{\beta_i}p_{j}^{\beta_j}}\equiv j+2(i-1)\pmod{p_i},
\]
and
\[
a_{p_{k}^{\beta_{k}}p_{\ell}^{\beta_{\ell}}}\pmod{p_\ell}\in\{3k-1, 3k, 3k+1\}, 
\]
observe that $p_i>j+2(i-1)>3i-2=3\ell-2\ge3k+1$, this contradicts $a_{p_{i}^{\beta_i}p_{j}^{\beta_j}} \pmod{ p_{i}^{\beta_i}p_{j}^{\beta_j}}$ and $a_{p_{k}^{\beta_{k}}p_{\ell}^{\beta_{\ell}}}\pmod{p_{k}^{\beta_{k}}p_{\ell}^{\beta_{\ell}}}$ overlap. The case $p_j = p_\ell$ can be treated in a similar way. It remains to discuss the case $p_j = p_\ell$. 

If $p_j=p_\ell$, we have $j=\ell$. Since 
\[
a_{p_{i}^{\beta_{i}}p_{j}^{\beta_{j}}}\pmod{p_j}\in\{3i-1,3i,3i+1\}
\]
and 
\[
a_{p_{k}^{\beta_{k}}p_{\ell}^{\beta_{\ell}}}\pmod{p_\ell}\in\{3k-1,3k,3k+1\}
\]
overlap. Since $3i + 1 < p_j$ and $3k + 1 < p_\ell$, we have $i=k$, this is a contradiction. 

Therefore, we have proved $A_1\bigcup A_2$ is good. 

We now construct congruence in $A_{(s)}$, where $3 \leq s\leq t$. Let $A_{(s)}=\mathop{\bigcup}\limits_{1\leq i_1<i_2<\cdots<i_s\leq t}A_{i_1,i_2,\cdots, i_s}$, where 
\[
A_{i_1,i_2,\cdots, i_s}=\left\{ a_{p_{i_1}^{\beta_{i_1}}p_{i_2}^{\beta_{i_2}}\cdots p_{i_s}^{\beta_{i_s}}}\pmod{p_{i_1}^{\beta_{i_1}}p_{i_2}^{\beta_{i+2}}\cdots p_{i_s}^{\beta_{i_s}}}: 1\leq\beta_{i_k}\leq\alpha_{i_k},1\leq k\leq s \right\}, 
\]
in which the residues $a_{p_{i_1}^{\beta_{i_1}}p_{i_2}^{\beta_{i_2}}\cdots p_{i_s}^{\beta_{i_s}}}$ satisfy the following $s$ restrictions
\[
\begin{cases}
a_{p_{i_1}^{\beta_{i_1}}p_{i_2}^{\beta_{i_2}}\cdots p_{i_s}^{\beta_{i_s}}} \pmod{p_{i_s}}
\in
\left\{a\pmod{p_{i_s}}: a\in A_{i_{s-1,i_s}}\right\}, \\
a_{p_{i_1}^{\beta_{i_1}}p_{i_2}^{\beta_{i_2}}\cdots p_{i_s}^{\beta_{i_s}}} \pmod{p_{i_{s-1}}}
\in
\left\{a\pmod{p_{i_{s-1}}}:a\in A_{i_{s-2},i_{s-1}}\right\}, \\
\quad \quad \quad \quad \quad \quad \quad \quad \quad \quad \quad \cdots\cdots \\
a_{p_{i_1}^{\beta_{i_1}}p_{i_2}^{\beta_{i_2}}\cdots p_{i_s}^{\beta_{i_s}}}\pmod{p_{i_2}} 
\in \left\{a\pmod{p_{i_2}}:a\in A_{i_1,i_2}\right\}, \\
a_{p_{i_1}^{\beta_{i_1}}p_{i_2}^{\beta_{i_2}}\cdots p_{i_s}^{\beta_{i_s}}}\pmod{p_{i_1}}
\in
\left\{a\pmod{p_{i_1}}:a\in A_{i_1,i_s}\right\}. 
\end{cases}
\]
Note that we have 
\[
\begin{cases}
\left\{a\pmod{p_{i_s}}: a\in A_{i_{s-1,i_s}}\right\} = \left\{3i_{s-1}-1,3i_{s-1} ,3i_{s-1}+1 \right\}, \\
\left\{a\pmod{p_{i_{s-1}}}:a\in A_{i_{s-2},i_{s-1}}\right\} = \left\{3i_{s-2}-1,3i_{s-2} ,3i_{s-2}+1 \right\}, \\
\quad \quad \quad \quad \quad \quad \quad \quad \quad \quad \quad \cdots\cdots \\
\left\{a\pmod{p_{i_2}}:a\in A_{i_1,i_2}\right\}
= \left\{3i_1-1,3i_1 ,3i_1+1 \right\}, \\
\left\{a\pmod{p_{i_1}}:a\in A_{i_1,i_s}\right\} = \left\{i_s+2(i_1-1)\right\}.
\end{cases}
\]
Thus, $a_{p_{i_1}^{\beta_{i_1}}p_{i_2}^{\beta_{i_2}}\cdots p_{i_s}^{\beta_{i_s}}}$ may be chosen from the set of $3^{s-1}$ residues modulo $p_{i_1}p_{i_2}\cdots p_{i_s}$. Clearly, when $s\geq3$ we have $3^{s-1}>2^s$. 
By Lemma~\ref{lemmanfactor}, we know that we can select the residues $a_{p_{i_1}^{\beta_{i_1}}p_{i_2}^{\beta_{i_2}}\cdots p_{i_s}^{\beta_{i_s}}}$ such that  $A_{i_1,i_2,\cdots, i_s}$ is good.

It is evident that $A_{(1)} \bigcup A_{i_1,i_2,\cdots, i_s}$ is good. Note that we have proved $A_{(1)}\bigcup A_{(2)}$ is good. Therefore, to prove $\mathop{\bigcup}\limits_{i = 1}^t A_{(i)}$ is good, by lemma~\ref{lemmangood} we only need to show $A_{i_1,i_2,\cdots ,i_s}\bigcup A_{j_1,j_2,\cdots ,i_k}$ is good for  $(i_1,i_2,\cdots,i_s)\neq(j_1,j_2,\cdots,j_k)$, where $3 \leq s \leq t$ and $2 \leq k \leq t$. 

Since $A_{i_1, i_2, \cdots, i_s}$ and $A_{j_1,j_2,\cdots ,i_k}$ are both good, if $A_{i_1,i_2,\cdots, i_s} \bigcup A_{j_1,j_2,\cdots, j_k}$ is not good, there exist two congruences
\[
a_{p_{i_1}^{\beta_{i_1}}p_{i_2}^{\beta_{i_2}}
\cdots p_{i_s}^{\beta_{i_s}}}
\pmod{p_{i_1}^{\beta_{i_1}}p_{i_2}^{\beta_{i_2}}
\cdots p_{i_s}^{\beta_{i_s}}}
\in 
A_{i_1,i_2,\cdots,i_s}
\]
and 
\[
a_{p_{j_1}^{\beta_{j_1}}p_{j_2}^{\beta_{j_2}}\cdots p_{j_k}^{\beta_{j_k}}}
\pmod{p_{j_1}^{\beta_{j_1}}p_{j_2}^{\beta_{j_2}}\cdots p_{j_k}^{\beta_{j_k}}}
\in A_{j_1,j_2,\cdots,j_k},
\]
that overlap with $\gcd{(p_{i_1}^{\beta_{i_1}}p_{i_2}^{\beta_{i_2}}\cdots p_{i_s}^{\beta_{i_s}},p_{j_1}^{\beta_{j_1}}p_{j_2}^{\beta_{j_2}}\cdots p_{j_k}^{\beta_{j_k}})}>1$. We set $p_{i_u}=p_{j_v}$ for some $1 \leq u \leq s$ and $1 \leq v \leq k$. 

When $i_u \neq i_1$ and $j_v \neq j_1$, we have 
\[
a_{p_{i_1}^{\beta_{i_1}}p_{i_2}^{\beta_{i_2}}\cdots p_{i_s}^{\beta_{i_s}}}\pmod{p_{i_{u}}}\in\left\{3i_{u-1}-1,3i_{u-1} ,3i_{u-1}+1  \right\}
\]
and
\[
a_{p_{j_1}^{\beta_{j_1}}p_{j_2}^{\beta_{j_2}}\cdots p_{j_k}^{\beta_{j_k}}}\pmod{p_{j_{v}}}\in\left\{3j_{v-1}-1,3j_{v-1} ,3j_{v-1}+1  \right\}. 
\]
Therefore, we have $i_{u-1}=j_{v-1}$, and we repeat this process to obtain $i_{u-2} = j_{v-2}$, $i_{u-3} = j_{v-3}$ and so on, until one of the the sub-index of $i$ and $j$ equals $1$. 

Without loss of generality we may assume $u \leq v$. Then we have $i_1=j_{v-u+1}$. Then 
\begin{equation}\label{modpi1}
a_{p_{i_1}^{\beta_{i_1}}p_{i_2}^{\beta_{i_2}}\cdots p_{i_s}^{\beta_{i_s}}}
\equiv i_s+2(i_1-1) \pmod{p_{i_{1}}}. 
\end{equation}
When $j_{v-u+1} \neq j_1$, we have 
\[
a_{p_{j_1}^{\beta_{j_1}}p_{j_2}^{\beta_{j_2}}\cdots p_{j_k}^{\beta_{j_{k}}}}\pmod{p_{j_{v-u+1}}}\in \left\{3j_{v-u}-1,3j_{v-u} ,3j_{v-u}+1 \right\}. 
\]
Since $j_{v-u+1}=i_1,i_s>i_1$, it follows that $3j_{v-u}+1<i_s+2(i_1-1)$, we have a contradiction. So we have $j_{v-u+1}=j_1=i_1$, i.e., we have $u = v$. In this case we have  
\begin{equation}\label{modpj1}
a_{p_{j_1}^{\beta_{j_1}}p_{j_2}^{\beta_{j_2}}\cdots p_{j_k}^{\beta_{j_k}}} \equiv j_k+2(j_1-1) \pmod{p_{j_1}}.    
\end{equation}
By Equations ~\eqref{modpi1} and \eqref{modpj1} we know $j_k=i_s$. We repeat the same process with $u = s$ and $v = k$ to obtain $(i_1,i_2,\cdots,i_s)=(j_1,j_2,\cdots,j_k)$, which is a contradiction. 

Therefore, we have proved $\mathop{\bigcup}\limits_{i = 1}^t A_{(i)}$ is good, which implies that $n$ is nice.

\noindent \textbf{Case 2.} Let $n = p_1^{\alpha_1} p_2^{\alpha_2} \cdots p_t^{\alpha_t}$, where $p_1< p_2<\cdots<p_t$, $\alpha_1 = 1$ and $\alpha_k\geq1$ for all $2 \leq k\leq t$. Let $t \leq p_1$. By Lemmas \ref{lemdivisor} and \ref{lemqpk}, we only need to prove the case where $t \geq 3$ and $\alpha_k \geq 2$ for all $2 \leq k \leq t$. 

Similar to Case 1, we construct good sets of congruences, $A_{(1)}$, $A_{(2)}$, $\cdots$, $A_{(t)}$, and prove $\mathop{\bigcup}\limits_{k = 1}^t A_{(k)}$ is a good set of congruences. For all $1 \leq k \leq t$, the set $A_{(k)}$ contains the congruences with moduli $m$ such that $m \mid n$ and $\omega(m) = k$. 

We first construct a good set of congruences $A_{(1)}$, in which the moduli have only one prime factor. We set
\[
A_{(1)} =
\left \{0 \pmod{p_1}\right\}
\bigcup\limits_{2 \leq i\leq t}
\left\{1 \pmod {p_i} \right\}
\bigcup\limits_{2 \leq i\leq t}
\left\{p_i^{\beta_i-1} \pmod {p_i^{\beta_i}}: 2 \leq \beta_i \leq \alpha_i\right\}.
\]
Similar to the proof of Lemma~\ref{lempk}, the set $A_{(1)}$ is good. Observe that for any $1 \leq i \leq t$, any congruence $a \pmod{p_i^{\beta_i}}$ in $A_{(1)}$, where $1 \leq \beta_i \leq \alpha_i$, satisfies $a\equiv0\pmod{p_1}$ or $a \equiv 0$ or $1 \pmod{p_i}$ for $2 \leq i\leq t$. 

Now, we construct the set of congruences $A_{(2)}$. We start by constructing subsets $A_{i,j}$ of $A_{(2)}$, where $2 \leq i < j \leq t$, as 
\[
A_{i,j}=\mathop{\bigcup}_{2\leq \beta_i\leq \alpha_i, 2 \leq \beta_j\leq \alpha_j}
\left\{ a_{p_{i}^{\beta_i}p_{j}^{\beta_j}} \pmod{ p_{i}^{\beta_i}p_{j}^{\beta_j}}
\right\}.
\]
Here, we require
\begin{equation}\label{1modpi}
a_{p_{i}^{\beta_i}p_{j}^{\beta_j}}\equiv j+2i-3 \pmod{p_i}
\end{equation}
and
\begin{equation}\label{1modpj}
a_{p_{i}^{\beta_i}p_{j}^{\beta_j}} \pmod{p_j} \in\{3i-2,3i-1,3i\}. 
\end{equation}
Since $t \leq p_1$, we have $t+2i-3 \leq p_1+2i-3\leq p_i-1$, thus the residues $a_{p_k^{\beta_k} p_l^{\beta_{\ell}}}\pmod{p_k}$ satisfy $\{2, 3, \cdots, t+2k-3\} \subset \{0, 1, \cdots, p_k - 1\}$ for all $2\leq k\leq t$. By Lemma~\ref{lemmatwofactor} with $a = j + 2i-3$, $b = 3i-2$, $c = 3i-1$, and $d = 3i $, there exists a good set of congruences $A_{i, j}$ satisfying Equations~\eqref{1modpi} and ~\eqref{1modpj}. Moreover, for $i=1$ and $2 \leq k \leq t$ we set
\[
A_{1,k}=\left\{ a_{p_{1}p_{k}^{\beta_{k}}}
\pmod{p_{1} p_{k}^{\beta_{k}}}:
1\leq\beta_{k}\leq\alpha_{k}\right\}. 
\]
We require
\[
a_{p_{1}p_{k}^{\beta_{k}}}\equiv k-1 \pmod{p_{1}}
\]
and
\[
a_{p_{1}p_{k}^{\beta_{k}}} \equiv p_k^{\beta_k-1}+2\pmod{p_k^{\beta_k}}.
\]
Then we have $a_{p_{1}p_{k}^{\beta_{k}}}\pmod{p_k}\in\{2,3\}$ when $2 \leq k \leq t$. We define 
\[
A_{(2)}=\bigcup_{1 \leq i<j \leq t}A_{i,j}. 
\]

Similar to the proof of Case 1, we have $A_{(1)}$, $A_{(1)} \bigcup A_{i,j}$ and $A_{i,j}\bigcup A_{k,\ell}$ are good for all $1 \leq i < j \leq t$ and $1 \leq k< \ell \leq t$ with $(i,j) \neq (k,\ell)$. By Lemma~\ref{lemmangood}, we conclude that $A_{(1)}\bigcup A_{(2)}$ is good.

We now construct congruence in $A_{(s)}$, where $3 \leq s\leq t$. Let $A_{(s)}=\mathop{\bigcup}\limits_{1\leq i_1<i_2<\cdots<i_s\leq t}A_{i_1,i_2,\cdots, i_s}$, where 
\[
A_{i_1,i_2,\cdots, i_s}=\left\{ a_{p_{i_1}^{\beta_{i_1}}p_{i_2}^{\beta_{i_2}}\cdots p_{i_s}^{\beta_{i_s}}}\pmod{p_{i_1}^{\beta_{i_1}}p_{i_2}^{\beta_{i_2}}\cdots p_{i_s}^{\beta_{i_s}}}: 1\leq\beta_{i_k}\leq\alpha_{i_k},1\leq k\leq s \right\}, 
\]
in which the residues $a_{p_{i_1}^{\beta_{i_1}}p_{i_2}^{\beta_{i_2}}\cdots p_{i_s}^{\beta_{i_s}}}$ satisfy the following $s$ restrictions
\begin{equation}\label{1restriction}
\begin{cases}
a_{p_{i_1}^{\beta_{i_1}}p_{i_2}^{\beta_{i_2}}\cdots p_{i_s}^{\beta_{i_s}}} \pmod{p_{i_s}}
\in
\left\{a\pmod{p_{i_s}}: a\in A_{i_{s-1,i_s}}\right\}, \\
a_{p_{i_1}^{\beta_{i_1}}p_{i_2}^{\beta_{i_2}}\cdots p_{i_s}^{\beta_{i_s}}} \pmod{p_{i_{s-1}}}
\in
\left\{a\pmod{p_{i_{s-1}}}:a\in A_{i_{s-2},i_{s-1}}\right\}, \\
\quad \quad \quad \quad \quad \quad \quad \quad \quad \quad \quad  \cdots\cdots \\
a_{p_{i_1}^{\beta_{i_1}}p_{i_2}^{\beta_{i_2}}\cdots p_{i_s}^{\beta_{i_s}}}\pmod{p_{i_2}} 
\in \left\{a\pmod{p_{i_2}}:a\in A_{i_1,i_2}\right\}, \\
a_{p_{i_1}^{\beta_{i_1}}p_{i_2}^{\beta_{i_2}}\cdots p_{i_s}^{\beta_{i_s}}}\pmod{p_{i_1}}
\in
\left\{a\pmod{p_{i_1}}:a\in A_{i_1,i_s}\right\}. 
\end{cases}
\end{equation}
Note that we have 
\begin{equation}\label{2restriction}
\begin{cases}
\left\{a\pmod{p_{i_s}}: a\in A_{i_{s-1,i_s}}\right\} = \left\{3i_{s-1}-2,3i_{s-1} -1,3i_{s-1} \right\}, \\
\left\{a\pmod{p_{i_{s-1}}}:a\in A_{i_{s-2},i_{s-1}}\right\} = \left\{3i_{s-2}-2,3i_{s-2}-1 ,3i_{s-2} \right\}, \\
\quad \quad \quad \quad \quad \quad \quad \quad \quad \quad \quad \cdots\cdots \\
\left\{a\pmod{p_{i_{3}}}:a\in A_{i_{2},i_{3}}\right\} = \left\{3i_{2}-2,3i_{2}-1 ,3i_{2} \right\}, \\
\left\{a\pmod{p_{i_2}}:a\in A_{i_1,i_2}\right\}
=\begin{cases}
\left\{3i_1-2,3i_1-1 ,3i_1 \right\},&\text{for} 
  \,i_1\neq 1, \\
\left\{ 2,3\right\},&\text{for} \,i_1=1,
\end{cases}  \\
\left\{a\pmod{p_{i_1}}:a\in A_{i_1,i_s}\right\} = \left\{i_s+2i_1-3\right\}.
\end{cases}
\end{equation}

Clearly, when $s\geq4$ we have $2\cdot3^{s-2}>2^s$. Thus, $a_{p_{i_1}^{\beta_{i_1}}p_{i_2}^{\beta_{i_2}}\cdots p_{i_s}^{\beta_{i_s}}}$ may be chosen from the set of at least $2\cdot3^{s-2}$ residues modulo $p_{i_1}p_{i_2}\cdots p_{i_s}$. 
By Lemma~\ref{lemmanfactor}, we know that we can select residues $a_{p_{i_1}^{\beta_{i_1}}p_{i_2}^{\beta_{i_2}}\cdots p_{i_s}^{\beta_{i_s}}}$ such that  $A_{i_1,i_2,\cdots, i_s}$ is good when $s \geq 4$.

It remains to investigate the construction when $s = 3$. Next we discuss the following two cases, based on whether $p_1 \nmid p_{i_1}p_{i_2}p_{i_3}$ or $p_1 \mid p_{i_1}p_{i_2}p_{i_3}$. 

When $p_1 \nmid p_{i_1}p_{i_2}p_{i_3}$, i.e., when $i_1 > 1$,  from the 3 restrictions \eqref{1restriction} and $\eqref{2restriction}$ of $a_{p_{i_1}^{\beta_{i_1}}p_{i_2}^{\beta_{i_2}}p_{i_3}^{\beta_{i_3}}}$, we can select $3^2>2^3$ distinct residues modulo $p_{i_1}p_{i_2}p_{i_3}$. By Lemma~\ref{lemmanfactor}, we can select $a_{p_{i_1}^{\beta_{i_1}}p_{i_2}^{\beta_{i_2}}p_{i_3}^{\beta_{i_3}}}$ such that  $A_{i_1,i_2,i_s}$ is good. 

When $p_1 \mid p_{i_1}p_{i_2}p_{i_3}$, then $i_1=1$.  
Note that  the exponent of $p_{i_1}$ is 1, 
thus $A_{i_1,i_2,i_3}
=\left\{ a_{{p_{i_1}}p_{i_2}^{\beta_{i_2}}p_{i_3}^{\beta_{i_3}}}
\pmod{p_{i_1}p_{i_2}^{\beta_{i_2}}p_{i_3}^{\beta_{i_3}}},
1\leq\beta_{i_k}\leq\alpha_{i_k} \right\}$. 
Depending on whether $\beta_{i_2} = 1$ or $\beta_{i_2} > 1$ and $\beta_{i_{3}} = 1$ or $\beta_{i_{3}} > 1$,  
there are $2^2=4$ possible cases. From the 3 restrictions \eqref{1restriction} and $\eqref{2restriction}$ of $a_{p_{i_1}p_{i_2}^{\beta_{i_2}}p_{i_3}^{\beta_{i_3}}}$,  
we can select $2\cdot3>2^2$ distinct residues modulo $p_{i_1}p_{i_2}p_{i_3}$.  
Then similar to the construction in Lemma~\ref{lemmanfactor}, 
we can select $a_{p_{i_1}p_{i_2}^{\beta_{i_2}}p_{i_3}^{\beta_{i_3}}}$ such that  $A_{i_1, i_2, i_3}$ is good. 

Similar to Case 1, we can prove that $\mathop{\bigcup}\limits_{i=1}^tA_{(i)}$ is good, which implies that  $n$ is nice. This completes the proof of Theorem~\ref{mainthm}. 

\section{Algorithm implementation and examples}
Having established the theoretical framework for hierarchical residue assignments, we present an algorithmic implementation of the proof of our main theorem. The algorithmic nature of our method allows us to systematically generate good congruence sets for any nice number. 

\tiny
\begin{algorithm}[H]
    \caption{Generating a good set of congruence of a nice number $n$}
    \begin{algorithmic}[]
        \STATE \textbf{Input}: $n$
        \STATE Factorize $n$: $n=p_1^{\alpha_1} \cdots p_t^{\alpha_t}$
        \STATE  \textbf{if} $t\geq p_1$
        \STATE \quad \textbf{return} n is not nice
        \STATE \textbf{else}
        \STATE \quad\textbf{if} $\alpha_1>1$ go to \textbf{Algorithm} Case 1
        \STATE \quad\textbf{else} go to \textbf{Algorithm} Case 2
        
    \end{algorithmic}
\end{algorithm}
\renewcommand{\thealgorithm}{}
\begin{algorithm}[H]
    \caption{Caes 1}
    \begin{algorithmic}[1]
    \STATE Initialize $A$ as an empty list
        \STATE \textbf{for }$i$ from 1 to $t$\\
        \STATE \quad Append $1\pmod {p_i}$ to $A$
        \STATE \quad\textbf{for} $j$ from 2 to $\alpha_i$\\
        \STATE\quad\quad Append $p_i^{j-1}\pmod {p_i^j}$ to $A$\\
        \STATE \quad \textbf{end for}
        \STATE \textbf{end for}
        \STATE \textbf{for} $i,j$ in $\{1,2,...,t\}$ with $i < j$ \textbf{do}
        \STATE\quad\textbf{for} $1\leq\beta_{i}\leq\alpha_i,1\leq\beta_{j}\leq\alpha_j$
        \STATE \quad\quad Apply Lemma~\ref{lemmatwofactor} with $ a=j+2(i-1),b=3i-1,c=3i$, and $d=3i+1$
        \STATE \quad\quad Append $a_{p_i^{\beta_i}p_j^{\beta_j}}\pmod{p_i^{\beta_i}p_j^{\beta_j}}$ to $A$
        \STATE \quad\textbf{end for}
        \STATE\textbf{end for}
        \STATE \textbf{for} $s$ from 3 to $t$ 
        \STATE \quad\textbf{for} $i_1,i_2,...,i_s$ in $\{1,2,...,t\}$ with $i_1<i_2<...<i_s$ \textbf{do}
        \STATE\quad\quad  Initialize $B_s$ as an empty list
        \STATE\quad\quad  \textbf{for} $k$ from 0 to $p_{i_1}p_{i_2}...p_{i_s}$
        \STATE\quad\quad\quad  \textbf{if}  $k\pmod {p_{i_s}}$ is in  $\{3i_{s-1}-1,3i_{s-1},3_{s-1}+1\}$, $\cdots$, $k\pmod {p_{i_{2}}}$  is in  $\{3i_{1}-1,3i_{1},3i_{1}+1\}$, and $k\pmod {p_{i_1}}$ is \\
 \quad \quad \quad \quad \quad in  $\{i_s+2(i_1-1)\}$
        \STATE\quad\quad\quad  \textbf{then}   Append $k$ to $S$
        \STATE\quad\quad\textbf{end for}
        \STATE\quad\quad Select $a_0,a_1,...,a_{2^s-1}$ in $S$
        \STATE\quad\quad Apply Lemma~\ref{lemmanfactor}
        \STATE \quad\quad\textbf{for} $1\leq\beta_{i_1}\leq\alpha_{i_1},\cdots ,1\leq\beta_{i_s}\leq\alpha_{i_s}$
        \STATE \quad\quad\quad Append $a_{ p_{i_1}^{\beta_{i_1}}p_{i_2}^{\beta_{i_2}}\cdots p_{i_s}^{\beta_{i_s}}}
                \pmod {p_{i_1}^{\beta_{i_1}}p_{i_2}^{\beta_{i_2}}\cdots p_{i_s}^{\beta_{i_s}}}$ to $A$
        \STATE\quad\quad \textbf{end for}
        \STATE\quad \textbf{end for}
        \STATE  \textbf{end for } 
        \STATE \textbf{return} $A$
    \end{algorithmic}
\end{algorithm}

\begin{algorithm}[H]
    \caption{Case 2}
    \begin{algorithmic}[1]
    \STATE Initialize $A$ as an empty list
    \STATE \textbf{while} $i_1,i_2,...,i_s$ in $\{2,3,...,t\}$ with $i_1<i_2<...<i_s$ and $s\geq1$, construct a set of congruences as in Case 1
    \STATE \textbf{while} $i_1=1$ and $s=1$ or $2$
    \STATE \quad Append $0\pmod{p_1}$ to $A$
    \STATE \quad\textbf{for} $j$ from 2 to $t$
    \STATE \quad\quad\textbf{for} $1\leq \beta_j\leq \alpha_j$
    \STATE\quad \quad \quad Append $\text{crt}(1\pmod{p_1},p_{j}^{\beta_{j}-1}\pmod{p_{j}^{\beta_{j}}})\pmod{p_1p_{j}^{\beta_{j}}}$ to $A$
    \STATE \quad\quad\textbf{end for}
    \STATE \quad\textbf{end for}
    \STATE \textbf{end while}
    \STATE \textbf{while} $i_1=1$ and $s\geq4$, construct a set of congruences as in Case 1
    \STATE \textbf{while} $i_1=1$ and $s=3$
    \STATE \quad Initialize $B_3$ as an empty list
    \STATE \quad\textbf{for} $k$ from 0 to $p_{i_1}p_{i_2}p_{i_3}$
    \STATE\quad\quad  \textbf{if}  $k\pmod {p_{i_3}}$ is in  $\{3i_{2}-2,3i_{2}-1,3_{2}\}$, $k\pmod {p_{i_{2}}}$ is in  $\{2,3\}$ and $k\pmod {p_{i_1}}$ is in  $\{i_3+2i_1-3\}$
    \STATE\quad\quad  \textbf{then}   Append $k$ to $S$
    \STATE\quad\textbf{end for}
    \STATE\quad Select $a_0,a_2,a_4,a_6$ in $B_3$ and apply Lemma~\ref{lemmanfactor}
    \STATE \quad\textbf{for} $1\leq\beta_{i_1}\leq\alpha_{i_1},1\leq\beta_{i_2}\leq\alpha_{i_2} ,1\leq\beta_{i_3}\leq\alpha_{i_3}$
    \STATE\quad \quad Append $a_{ p_{i_1}^{\beta_{i_1}}p_{i_2}^{\beta_{i_2}}p_{i_3}^{\beta_{i_3}}}
                \pmod {p_{i_1}^{\beta_{i_1}}p_{i_2}^{\beta_{i_2}}p_{i_3}^{\beta_{i_3}}}$ to $A$
    \STATE \quad\textbf{end for}
    \STATE \textbf{end while}
    \STATE \textbf{return} $A$
    \end{algorithmic}
\end{algorithm} 

\normalsize 

\noindent \textbf{Input 1}: $n = 5^2 \cdot 7^2 \cdot 11^2 \cdot 13^2$. 

\noindent \textbf{Output 1}: $n$ is a nice number. A good set of congruences corresponding to $n$ consists of the following $80$ congruences. 
\begin{align*}
& x \equiv  1 \pmod{5}, &
& x \equiv  1 \pmod{7}, & \\
& x \equiv  1 \pmod{11}, &
& x \equiv  1 \pmod{13}, &\\
& x \equiv  5 \pmod{25},&
& x \equiv  2 \pmod{35}, &\\
& x \equiv  7 \pmod{49},&
& x \equiv  13 \pmod{55}, &\\
& x \equiv  54 \pmod{65},&
& x \equiv  5 \pmod{77}, &\\
& x \equiv  83 \pmod{91},&
& x \equiv   11 \pmod{121},& \\
& x \equiv    8 \pmod{143},&
& x \equiv   13 \pmod{169}, &\\
& x \equiv  137 \pmod{175},&
& x \equiv   207 \pmod{245}, &\\
& x \equiv  58 \pmod{275},&
& x \equiv   159 \pmod{325}, &\\
& x \equiv   93 \pmod{385}, &
& x \equiv   44 \pmod{455}, &\\
& x \equiv  61 \pmod{539},&
& x \equiv   378 \pmod{605}, &\\
& x \equiv  552 \pmod{637},&
& x \equiv  684 \pmod{715},& \\
& x \equiv  524 \pmod{845},&
& x \equiv  502 \pmod{847}, &\\
& x \equiv  489 \pmod{1001},&
& x \equiv   20 \pmod{1183},& \\
& x \equiv  557 \pmod{1225},&
& x \equiv  503 \pmod{1573}, &\\
& x \equiv 1713 \pmod{1859},&
& x \equiv  633 \pmod{1925}, &\\
& x \equiv  864 \pmod{2275},&
& x \equiv  513 \pmod{2695}, &\\
& x \equiv 2083 \pmod{3025},&
& x \equiv  709 \pmod{3185},& \\
& x \equiv 1114 \pmod{3575},&
& x \equiv 2734 \pmod{4225}, &\\
& x \equiv  688 \pmod{4235},&
& x \equiv 2634 \pmod{5005}, &\\
& x \equiv  604 \pmod{5915},&
& x \equiv 5353 \pmod{5929}, &\\
& x \equiv 1777 \pmod{7007},&
& x \equiv  959 \pmod{7865}, &\\
& x \equiv 1385 \pmod{8281},&
& x \equiv 1179 \pmod{9295}, &\\
& x \equiv 1336 \pmod{11011},&
& x \equiv 1413 \pmod{13013}, &\\
& x \equiv  753 \pmod{13475},&
& x \equiv  564 \pmod{15925},& \\
& x \equiv 17443 \pmod{20449},&
& x \equiv  458 \pmod{21175},& \\
& x \equiv 9214 \pmod{25025},&
& x \equiv  514 \pmod{29575}, &\\
& x \equiv  423 \pmod{29645},&
& x \equiv 5924 \pmod{35035}, &\\
& x \equiv 1269 \pmod{39325},&
& x \equiv  759 \pmod{41405}, &\\
& x \equiv  894 \pmod{46475},&
& x \equiv 6484 \pmod{55055}, &\\
& x \equiv 9564 \pmod{65065},&
& x \equiv 1217 \pmod{77077},& \\
& x \equiv 1700 \pmod{91091},&
& x \equiv 1334 \pmod{102245}, &\\
& x \equiv 1854 \pmod{143143},&
& x \equiv  578 \pmod{148225}, &\\
& x \equiv 9719 \pmod{175175},&
& x \equiv  669 \pmod{207025}, &\\
& x \equiv 5274 \pmod{275275},&
& x \equiv 8354 \pmod{325325}, &\\
& x \equiv 5364 \pmod{385385},&
& x \equiv 7849 \pmod{455455}, &\\
& x \equiv 1049 \pmod{511225},&
& x \equiv 9004 \pmod{715715},& \\
& x \equiv 1140 \pmod{1002001},&
& x \equiv 8564 \pmod{1926925}, &\\
& x \equiv 6639 \pmod{2277275},&
& x \equiv 7199 \pmod{3578575}, &\\
& x \equiv 7289 \pmod{5010005},&
& x \equiv 6079 \pmod{25050025}. &
\end{align*}

\noindent \textbf{Input 2}: $n = 3 \cdot 5^2 \cdot 7^2$. 

\noindent \textbf{Output 2}: $n$ is a nice number. A good set of congruences corresponding to $n$ consists of the following $17$ congruences. 
\begin{align*}
&x \equiv    0 \pmod{3}, &
&x \equiv    1 \pmod{5}, &\\
&x \equiv    1 \pmod{7}, &
&x \equiv   13 \pmod{15},& \\
&x \equiv   17 \pmod{21}, &
&x \equiv    5 \pmod{25}, &\\
&x \equiv    4 \pmod{35},&
&x \equiv    7 \pmod{49}, &\\
&x \equiv    7 \pmod{75}, &
&x \equiv   32 \pmod{105}, &\\
&x \equiv  107 \pmod{147},&
&x \equiv  159 \pmod{175}, &\\
&x \equiv   209 \pmod{245},&
&x \equiv  167 \pmod{525}, &\\
&x \equiv  152 \pmod{735},&
&x \equiv  559 \pmod{1225}, &\\
&x \equiv  158 \pmod{3675}.&
&&
\end{align*}

\noindent \textbf{Input 3}: $n = 3^2 \cdot 5^2 \cdot 7^2$. 

\noindent \textbf{Output 3}: $n$ is not a nice number. 

\section*{Data availability}
Data will be made available upon request.

\section*{Declaration of Competing Interest}
The authors declare that there is no conflict of interest.

\section*{Acknowledgments}
Huixi Li's research is supported by the National Natural Science Foundation of China (Grant No.12201313). This research was carried our during a visit by Huixi Li to the Chern Institute of Mathematics. The authors are grateful to the Chern Institute of Mathematics for providing a wonderful working environment.   

\bibliographystyle{plain}
\bibliography{bib}

\end{document}